\documentclass[a4paper, 12pt]{article}

\usepackage[T2A]{fontenc}
\usepackage[cp1251]{inputenc}
\usepackage[english]{babel}
\usepackage{amsmath, amssymb, amsfonts, amsthm}

\usepackage{bxeepic}

\usepackage[text={175mm,260mm},nofoot,centering]{geometry}

\theoremstyle{plain}
\newtheorem{theorem}{Theorem}
\newtheorem{lemma}{Lemma}

\makeatletter
\renewcommand{\@oddhead}{\vbox{\hbox to\textwidth{\hfil\small The partial gossip problem revisited\hfil\thepage}\hrule}}%
\makeatother

\let\leq\leqslant
\let\geq\geqslant

\begin{document}

\title{The partial gossip problem revisited}

\author{O.~Bursian, K.~Kokhas}

\maketitle

\begin{abstract}
We present correct proof of Chung~G., Tsay Y.-J. result on partial gossip problem.
\end{abstract}


We consider the following variation of gossip problem.

Let $n$ and $k$ be arbitrary natural numbers, $2\leq k\leq n$. There are $n$ persons, each of whom knows a message (a gossip). A pair of persons can pass all messages they have by making a telephone call. The partial gossiping problem is to determine the minimum number of calls needed for all person to know at least $k$ messages. Denote this minimum number of calls by  $P(n,k)$.

The numbers $P(n,k)$ were found by Chung~G. and Tsay Y.-J., see \cite{ChungTsay}, where the following theorem is presented. Fix $k$ and consider finite sequence 
\begin{equation}
\label{eqn:seq}
t_i(k)=i+2^{k-i-2}, \quad \text{where } -1\leq i\leq k-4.
\end{equation}
We write often $t_i$ instead of $t_i(k)$ for short. The sequence $t_i$ decreases and satisfies the inequality $2t_i-i>t_{i-1}$.

\begin{theorem}\label{thm:chung-tsay}
{\rm 1)} $P(n,k)=\bigl\lceil\frac{2^{k-1}-1}{2^{k-1}}\cdot n\bigr\rceil$ for $n\geq 2^{k-1}-1=t_{-1}$.

{\rm 2)} If $t_i\leq n< t_{i-1}$, where $0\leq i\leq k-4$, then $P(n,k)= n+i$.
\end{theorem}

Both statements of the theorem are correct but unfortunately the proof of the second statement in \cite{ChungTsay} contains an error: an <<evident>> statement in parentheses in \cite[Lemma 3]{ChungTsay} is wrong, see comment after the proof of lemma~\ref{lem:1} below. In this paper, we prove the second statement of theorem \ref{thm:chung-tsay}. We follow the general lines of~\cite{ChungTsay} but gossip spreading graphs turned out to be more diversified.

\subsubsection*{Lemmas}

\emph{Communication scheme} is a finite sequence of calls between persons. A \emph{communication graph} is a multigraph whose vertices are persons and edges are calls. We may think that edges in communication graph are numbered chronologically. Communication schemes and graphs are denoted by the same letters. For each person $A$ we call gossip $A$ the message he knows at the beginning.

We say that a person is \emph{$\ell$-informed} if he knows at least $\ell$ messages when all the calls are completed, the number $\ell$ is called the \emph{awareness} of the person. A communication graph is \emph{$\ell$-informing} if it makes all persons $\ell$-informed. \emph{Exact $\ell$-informing} graph is a communication graph such that each person knows exactly $\ell$ messages when all the calls are completed. For any subset~$C$ of initial set of calls we may consider a subgraph of the communication graph generated by $C$, it contains edges corresponding to calls from $C$ and the set of its vertices is the set of endpoints of these edges. 

Connected graph in which the number of vertices equals to the number of edges is called \emph{unicyclic}.
Each unicyclic graph contains exactly one simple cycle or an edge of multiplicity 2 as degenerate case.

\begin{lemma}\label{lem:1}
Let $n$ be the number of persons in the communication scheme.

a{\rm)} Let $n\geq k\geq 1$. If communication graph is a $k$-informing tree then $n\geq 2^{k-1}$.

b{\rm)} Let $n\geq k\geq k'\geq 1$ and communication graph be a tree. Let one of the persons be $k'$-informed and the others be $k$-informed when all the calls are completed. Then
$$
n\geq (\underbrace{2^{k-2}+2^{k-3}+2^{k-4}+\ldots}_{\text{$k'-1$ \rm summands}}) + 1.
$$

с{\rm)} Let $n\geq k\geq 4$ and communication graph be unicyclic and $k$-informing. Then $n\geq 2^{k-2}$.
\end{lemma}

\proof
Observe that lemma statement allows degenerate case $k=1$ (or $k'=1$) when the tree consists of one vertex only.

a) Induction by $k$. If we remove the first call, the tree splits onto two components, which are $(k-1)$-informing trees. Apply the induction hypothesis to these components.

b) Induction by $k$, too. If we remove the first call, the tree splits onto two componets. One of them is a $(k-1)$-informing tree. In the other component one person is $(k'-1)$-informed and the others are $(k-1)$-informed. Applying statement a) to the first component and the induction hypothesis to second we obtain the proof of induction step.

c) When we remove the first call, the obtained graph consists of one or two components, one of which is a $(k-1)$-informing tree. By the statement a) this component has at least $2^{k-2}$ vertices.
\endproof

Lemma \ref{lem:1}\,а) is \cite[Lemma 3]{ChungTsay}, we cite the proof from this article. Lemma \ref{lem:1}\,b) is a corrected second statement of \cite[Lemma 3]{ChungTsay}. The original proof contained a mistake: if one of the graph components after removal of the first call contains only one vertex, namely the less informed person, we can not use $k$ in the induction step correctly as the <<awareness rate>> of persons in that component. Due to this inaccuracy the authors of \cite{ChungTsay} missed variable examples of communication schemes such as examples 2 and 3 below and fig.~\ref{ris:kk-method}, \ref{ris:kk-ext-method}.

\smallskip
In the next lemmas we call the calls of a given communication scheme $G'$ \emph{base} and we enlarge scheme $G'$ by appending several \emph{preliminary} calls that should be performed before the main calls. We denote the obtained communication scheme by $G$ and call it \emph{enlarged communication scheme}.

\begin{lemma}\label{lem:prirost_osved}
Let enlarged communication scheme $G$ be obtained by appending $\ell$ preliminary calls to a communication scheme $G'$. Then awareness of each person has increased at most by $\ell$.
\end{lemma}

\proof 
Let the number of messages known to person $A$ after all communications of scheme $G'$ is by $\ell+1$ more (due to preliminary calls) than those in scheme $G$. Then at the beginning these additional gossips were known to some persons $B_1$, \dots, $B_{\ell+1}$, who participated in the preliminary calls. For each message $B_i$ consider the path in the communication graph which delivers this message to $A$ (if there are several paths, choose one of the shortest paths). It is clear that each path starts with a preliminary call. Since we have $\ell +1$ paths, two of them start with the same preliminary call. This is impossible due to minimality.
\endproof

An increase of persons' awareness by $\ell$ as a result of $\ell$ preliminary calls is possible even for the case where the initial communication scheme was a $k$-informing tree. For $\ell=1$ this tree can even be \emph{minimal}, i.e., contain $2^{k-1}$ vertices (see Fig.~\ref{ris:kk-derevo+1sluh}). As the following lemmas show, for $\ell>1$ the tree can no longer be minimal  (see Fig.~\ref{ris:kk-derevo+2sluha}).

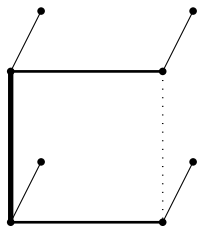
\begin{figure}[t]
\footnotesize 
\begin{center}
\setlength{\unitlength}{1mm}
\qquad\begin{picture}(95,35)(70,-26)
\multiput(70,0)(20,0){2}{\multiput(0,0)(0,-20){2}{\circle*{1}}}
\matrixput(74,8)(20,0){2}(0,-20){2}{\circle*{1}}
\linethickness{.4pt}
\matrixput(70,0)(20,0){2}(0,-20){2}{\put(0,0){\line(1,2){4}}}
\dottedline{1}(90,0)(90,-20)
\linethickness{1pt}
\multiput(70,0)(0,-20){2}{\put(0,0){\line(1,0){20}}}
\linethickness{2pt}
\put(70,0){\line(0,-1){20}}
\end{picture}
\end{center}
\caption{\vtop{\hsize=370pt 
In this example the communication graph is a minimal tree (solid lines, the thicker line the earlier call), that makes all persons $k$-informed (${k=4}$). Adding one preliminary call (dash line) makes the persons $(k+1)$-informed.
}}
\label{ris:kk-derevo+1sluh}
\end{figure}

\begin{figure}[t]
\footnotesize
\begin{center}
\setlength{\unitlength}{.7mm}
\qquad\qquad\qquad
\begin{picture}(90,45)(30,-26)
\matrixput(70,0)(20,0){2}(0,-20){2}{\circle*{2}}
\matrixput(74,8)(20,0){2}(0,-20){2}{\circle*{2}}
\matrixput(50,0)(0,-20){2}(-20,0){2}{\circle*{2}}
\linethickness{.4pt}
\matrixput(70,0)(20,0){2}(0,-20){2}{\put(0,0){\line(1,2){4}}}
\dottedline{1}(70,0)(70,-20)
\dottedline{1}(50,0)(50,-20)
\linethickness{1pt}
\matrixput(30,0)(40,0){2}(0,-20){2}{\put(0,0){\line(1,0){20}}}
\linethickness{2pt}
\multiput(50,0)(0,-20){2}{\put(0,0){\line(1,0){20}}}
\linethickness{3pt}
\put(30,-20){\line(0,1){20}}
%
\end{picture}
\qquad
\begin{picture}(90,65)(30,-40)
\matrixput(70,0)(20,0){2}(0,20){2}{\circle*{2}}
\matrixput(50,0)(0,-20){3}(-20,0){2}{\circle*{2}}
\linethickness{.4pt}
\multiput(70,0)(0,20){2}{\put(0,0){\line(1,0){20}}}
\multiput(30,-20)(20,0){2}{\put(0,0){\line(0,-1){20}}}
\put(30,0){\line(1,0){20}}
\dottedline{1}(30,0)(70,20)
\dottedline{1}(50,0)(50,-20)
\linethickness{1pt}
\put(70,0){\line(0,1){20}}
\linethickness{2pt}
\multiput(30,-20)(20,20){2}{\put(0,0){\line(1,0){20}}}
\linethickness{3pt}
\put(30,-20){\line(0,1){20}}
%
\end{picture}
\end{center}

\caption{\vtop{\hsize=370pt 
In these two exmples the communication graphs are trees (solid lines, the thicker line the earlier call), that makes all persons $k$-informed (${k=4}$). Adding two preliminary calls (dash lines) makes the persons $(k+2)$-informed.
}}
\label{ris:kk-derevo+2sluha}
\end{figure}

\begin{lemma}\label{lem:k -> k+l exact}
Let $n> k\geq 3$, $\ell \geq 1$. Let $G'$ be an exact $k$-informing tree on $n$ vertices. Let the persons make $\ell$ preliminary calls, and after that make all calls according to scheme $G'$. Assume that these calls make all the persons $(k+\ell)$-informed. Then
$$
n\geq t_{\ell-1}(k+\ell)=2^{k-1}+\ell-1.
$$
\end{lemma}

\proof
Denote by $L$ the graph of preliminary calls.

Prove that graph $L$ is a union of non-intersecting edges. By lemma~\ref{lem:prirost_osved} after adding of $\ell$ preliminary calls the awareness of any person increases at most by $\ell$, therefore the final $(k+\ell)$-awareness is exact, too:
\begin{equation}
\label{eqn:rovno po i}
\text{due to preliminary calls each person learns exactly $\ell$ \emph{new} gossips.}
\end{equation}

Let $A$ be an arbitrary person, $L'$ be a connected component of graph $L$, $\ell'$ be the number of edges in $L'$. Then $L'$ contains at most $\ell'+1$ vertices, i.e. at most $\ell'+1$ gossips are circulating in all the calls of the component $L'$. Then in the enlarged communication scheme person $A$ learned at most $\ell'$ new gossips from $L'$. Indeed, person $A$ can learn a new gossip from component $L'$ only if either $A$ belongs to $L'$ or the gossip from $L'$ was delivered to $A$ by the chain of calls from $G'$. In the latter case the gossip in $L'$ nearest to $A$ (in that chain) is not new for $A$. And so in both cases $A$ learns at most $\ell'$ new gossips.

Thus, for each person the number of his new gossips from any component of graph $L$ does not exceed the number of edges in this component. Due to \eqref{eqn:rovno po i} we conclude that each person learned from each connected component of $L$ exactly as many new gossips as the number of edges in it. Therefore the number of vertices in each component equals the number of edges plus 1, i.e. each component is a tree. It remains to observe that if a tree contains more than 1 edge then there exists a person in a leaf vertex that did not take part in the last call in this tree. Then this person learned less than $\ell'$ new gossips from this tree, a contradiction. Thus, each connected component of graph $L$ is a separate edge.

Let's note one more detail of gossips spreading. Let edge $PQ$ be a connected component of $L$. For each person in $G$ one of gossips $P$ and $Q$ is new and the other is ``old''. Hence gossips $P$ and $Q$ both spread throughout graph $G$ during the enlarged communication scheme.

Now we will show that the conditions of lemma are possible for $\ell\leq 3$ only. Let us introduce some notations. Let $A$ and $B$ be the persons who made the first base call in $G'$. Let after removing of edge $AB$ graph $G'$ split into two trees $\mathcal A$ and $\mathcal B$, where $A\in\mathcal A$, $B\in\mathcal B$ (see Fig.~\ref{ris:kk_B=E}a). Let $C$ and $D$ be the persons who made the first base call in~$\mathcal A$ (since $k\geq 3$ these persons exist), and $E$ and $F$ be the persons who made the first base call in~$\mathcal B$ (for the sake of definiteness let $C$ is closer to $A$ than $D$ and $E$ is closer to $B$ than $F$). By $\mathcal A_C$ and $\mathcal A_D$ we denote the sub-trees obtained from tree $\mathcal A$ by removing edge $CD$, and by $\mathcal B_E$ and $\mathcal B_F$ we denote the sub-trees obtained from tree $\mathcal B$ by removing edge $EF$ The sub-trees $\mathcal A_C$, $\mathcal A_D$, $\mathcal B_E$, $\mathcal B_F$ are connected in graph $G'$ by edges  $AB$, $CD$, $EF$ only. We consider several cases.

1. The six persons mentioned above are different (Fig.~\ref{ris:kk_B=E}a). Then each edge of graph $L$ is incident to one of vertices  $A$, $B$, $C$, $D$, $E$, $F$. Indeed, if edge $XY$ of graph $L$ contains none of the six vertices above, then gossips $X$ and $Y$ can not spread throughout all four components $\mathcal A_C$, $\mathcal A_D$, $\mathcal B_E$, $\mathcal B_F$ since these gossips were not pass in calls $AB$, $CD$, $EF$. Moreover, if $X$ and/or $Y$ coincides with one of $A$, $B$, $C$, $D$, $E$, $F$, then at most two of edges $AB$, $CD$, $EF$ were used for spreading gossips $X$ and $Y$ and hence these calls can spread throughout all four components only if $X\in \{C,D\}$, $Y\in\{E,F\}$ (or vice versa). Thus in the case under consideration graph $L$ contains at most two edges.

\begin{figure}[h]
\footnotesize
\begin{center}
\setlength{\unitlength}{.7mm}
\begin{picture}(80,68)(14,-42)
\matrixput(70,0)(0,-20){2}(-20,0){3}{\circle*{2}}
\allinethickness{.4pt}
\matrixput(0,-3)(43,0){2}(0,-34){2}{\put(37,10){\oval(32,25)}}
\put(58,8){\oval(94,32)}\put(58,-28){\oval(94,32)}
\linethickness{2pt}
\multiput(50,-20)(0,20){2}{\put(0,0){\line(1,0){20}}}
\linethickness{3pt}
\put(30,-20){\line(0,1){20}}
\put(23,3){$A$}\put(23,-25){$B$}
\put(45,3){$C$}\put(45,-25){$E$}
\put(70,3){$D$}\put(71,-25){$F$}
\put(13,14){$\mathcal A_C$}\put(14,-40){$\mathcal B_E$}
\put(96,14){$\mathcal A_D$}\put(95,-40){$\mathcal B_F$}
\put(5,3){$\mathcal A$}\put(5,-22){$\mathcal B$}
\end{picture}
\hfil\hfil\hfil
\begin{picture}(80,65)(14,-40)
\matrixput(70,0)(0,-20){2}(-40,0){2}{\circle*{2}}
\put(50,0){\circle*{2}}
\allinethickness{.4pt}
\matrixput(0,-3)(43,0){2}(0,-34){2}{\put(37,10){\oval(32,25)}}
\put(58,8){\oval(94,32)}\put(58,-28){\oval(94,32)}
\linethickness{2pt}
\put(30,-20){\line(1,0){40}}\put(50,0){\line(1,0){20}}
\linethickness{3pt}
\put(30,-20){\line(0,1){20}}
\put(23,3){$A$}\put(23,-25){$B=E$}
\put(45,3){$C$}
\put(70,3){$D$}\put(71,-25){$F$}
\put(13,14){$\mathcal A_C$}\put(14,-40){$\mathcal B_E$}
\put(96,14){$\mathcal A_D$}\put(95,-40){$\mathcal B_F$}
\put(5,3){$\mathcal A$}\put(5,-22){$\mathcal B$}
\end{picture}
\end{center}

\centerline{\hfill a) Case 1\hfill\hfill b) Case 2\hfill}
\caption{First calls in $k$-informing tree}
\label{ris:kk_B=E}
\end{figure}
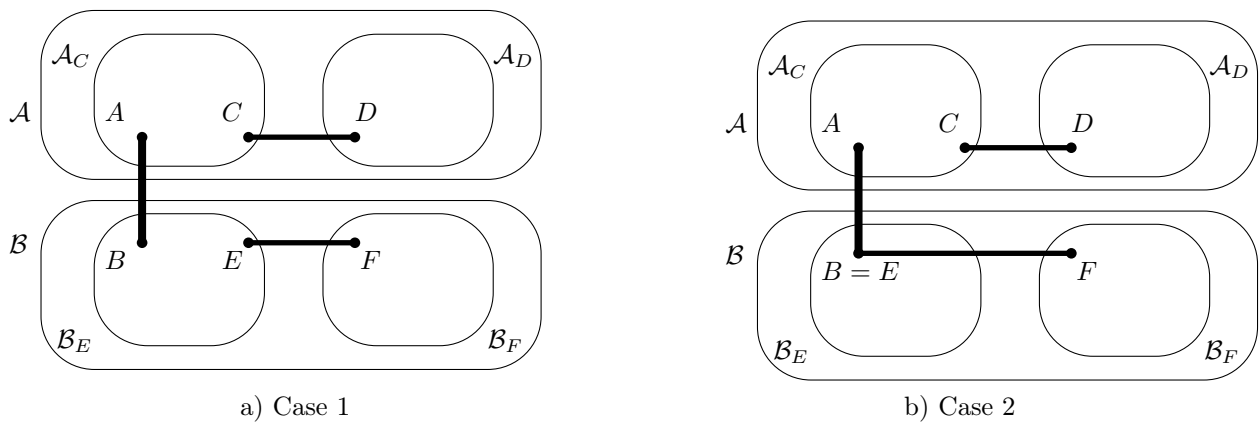

2. $B=E$, $A\ne C$ (the case $A=C$, $B\ne E$ is similar). In this case each edge of graph $L$ is incident to one of vertices  $A$, $B$, $C$, $D$, $F$. 

Subcase 1. Graph $L$ does not contain edge $AC$. If $L$ contains an edge $CX$ or $DX$, then $X\in\{B,F\}$ because otherwise the new gossip of this edge does not go through bridges $AB$, $BF$ and can not spread throughout the whole graph. If $L$ contains an edge $AZ$, $BZ$ or $FZ$, then $Z\in \{C\}\cup\mathcal A_D$ since otherwise the new gossip from this edge does not arrive to $\mathcal A_D$. From the other hand, if $Z\in \mathcal A_D\setminus\{D\}$, then $C$ does not learn gossip $Z$. Hence $Z\in \{C,D\}$ for edges under discussion. Thus graph $L$ contains at most 2 edges in this subcase.

Subcase 2. Graph $L$ contains edge $AC$. If $L$ contains an edge $DX$, then $X\in\{B,F\}$, since otherwise the new gossip does not arrive to $\mathcal B_F$. Now only one vertex remains, that can be incident with edge of graph $L$, i.e. $L$ has at most 3 edges. And if $L$ does not contain edge $DX$, then vertices $B$ and $F$ remain and $L$ has at most 3 edges once again.

3. $A=C$, $B=E$. By analogous reasoning we conclude that $L$ can only have edges $DF$, $AX$ and $BY$ for some $X$, $Y$.
Thus graph $L$ contains at most 3 edges in this case.

\begin{figure}[h]
\footnotesize
\begin{center}
\setlength{\unitlength}{.7mm}
\begin{picture}(80,65)(14,-40) 
\matrixput(70,0)(0,-20){2}(-40,0){2}{\circle*{2}}
\allinethickness{.4pt}
\matrixput(0,-3)(43,0){2}(0,-34){2}{\put(37,10){\oval(32,25)}}
\put(58,-28){\oval(94,32)}\put(5,-22){$\mathcal B$}
\linethickness{2pt}
\put(30,-20){\line(1,0){40}}\put(30,0){\line(1,0){40}}
\linethickness{3pt}
\put(30,-20){\line(0,1){20}}
\put(23,3){$A=C$}\put(23,-25){$B=E$}
\put(70,3){$D$}\put(71,-25){$F$}
\put(13,14){$\mathcal A_C$}\put(14,-40){$\mathcal B_E$}
\put(96,14){$\mathcal A_D$}\put(95,-40){$\mathcal B_F$}
\end{picture}
\hfil\hfil\hfil
\begin{picture}(80,65)(14,-40) 
\matrixput(70,0)(0,-20){2}(-40,0){2}{\circle*{2}}
\allinethickness{.4pt}
\matrixput(0,-3)(43,0){2}(0,-34){2}{\put(37,10){\oval(32,25)}}
%
\put(24,3){$ABD$}\put(24,-26){$ABF$}
\put(70,3){$ABD$}\put(70,-26){$ABF$}
\put(40,12){\circle*{2}}\put(41.5,13){$X$}
\put(13,14){$\mathcal A_C$}\put(14,-40){$\mathcal B_E$}
\put(96,14){$\mathcal A_D$}\put(95,-40){$\mathcal B_F$}
\end{picture}
\end{center}

\centerline{\hfill a) \hfill\hfill b) After the first 3 calls\qquad\qquad}
\caption{Case 3}
\label{ris:case3}
\end{figure}

It remains to prove lemma for $\ell\leq 3$. In cases 1 and 2 (Fig.~\ref{ris:kk_B=E}a,b) subgraphs $\mathcal B$ and $\mathcal A_D$ are $(k-1)$-informing trees with at least $2^{k-2}$ vertices in each of them by Lemma~\ref{lem:1}a. Besides that, graph $G'$ contains vertices $A$ and $C$, hence
\begin{equation}
n\geq 2^{k-1}+2, 
\label{eq:min plus2}
\end{equation}
this inequality is stronger than the desired one. In case 3, if $L$ contains one edge only, the desired inequality coincides with the inequality of Lemma~\ref{lem:1}a.

We now need to analyze the part of Case 3 where the scheme $L$ contains 2 or 3 calls. Without loss of generality, we can assume that one of them is the call $AX$. As a result of this call, person~$X$ must learn a new message from $A$ (new compared to what he learns from the calls in~$G'$). Therefore, in scheme $G'$, gossip $A$ (as well as gossip~$B$) did not reach vertex $X$. Similarly, gossip~$X$ did not reach vertices $A$ and $B$.

Let us consider the distribution of messages after the first three calls of scheme $G'$ in Case 3 (see Fig.~\ref{ris:case3}b; the messages known to each are written near the highlighted vertices). For definiteness, we have placed vertex $X$ in subgraph $\mathcal A_C$ (the reasoning is analogous for other subgraphs). In this component, messages from vertex ``ABD'' do not reach vertex $X$.

In communication scheme $L$, there is at least one more preliminary call — $BY$ and/or $DF$. Note that preliminary call $DF$ is impossible in this situation because vertex $X$ would not be able to learn anything new from it. Hence, $\ell=2$ and scheme $L$ contains only the calls $AX$ and $BY$.

Similar to vertex $X$, vertex $Y$ did not receive messages from $A$ and $B$ in scheme $G'$. However, vertex $X$ must learn a new message during the call $BY$. Gossip $Y$ cannot be the new one because it is ``attached'' to gossip $B$, and gossip $B$ does not reach $X$ in scheme $G'$. Therefore, this new message is gossip $B$, and vertex $X$ can learn this message only if messages from $Y$ reach $X$ in scheme~$G'$. Similarly, messages from $X$ reach $Y$ in scheme $G'$ as well. This is possible only if vertices $X$ and $Y$ are adjacent in graph $G'$ (in which case vertex $Y$ is also located in component~$\mathcal A_C$).

Now, components $\mathcal A_C$ and $\mathcal B$ are $(k-1)$-awareness trees, and by Lemma~\ref{lem:1}a, each contains at least $2^{k-2}$ vertices. Additionally, there is vertex $D$, which has not been counted in these totals. Together with it, we obtain the required inequality $n\geq 2^{k-1}+1$.
\endproof

\begin{lemma}\label{lem:k -> k+l}
a) Let $n\geq k\geq 4$, $0\leq i\leq k-4$. Let an enlarged communication scheme $G$ for $n$ persons be obtained by appending $i$ preliminary calls to scheme $G'$, where graph $G'$ is a tree on $n$ vertices. If $G$ is a $k$-informing scheme, then
$$
n\geq t_{i-1}(k)=2^{k-i-1}+i-1.
$$

b) Let $n\geq k\geq 4$, $0\leq i\leq k-4$. Let an enlarged communication scheme $G$ for $n$ persons be obtained by appending $i$ preliminary calls to scheme $G'$, where graph $G'$ is a tree. If scheme $G$ makes all vertices of graph $G'$ $k$-informed then $n\geq t_{i-1}(k)$.
\end{lemma}

The difference between the parts of the lemma is the following: in part (a), the preliminary $i$ calls are made between vertices from $G'$, while in part (b) these calls may involve ``outsiders'' (the parameter $n$ in this case denotes the total number of all vertices, not just those belonging to the tree $G'$). For clarity, we emphasize the distinction between the properties: ``graph $G'$ is $k$-informing'' (the required awareness is achieved solely through calls within $G'$) and ``all gossips from $G'$ have become $k$-informed'' (calls may also occur with vertices outside $G'$).

\proof  We prove both parts by descent on the parameter $n$ to the case described in Lemma~\ref{lem:k -> k+l exact}.

a) By Lemma \ref{lem:prirost_osved}, the awareness of persons could not have increased by more than $i$ as a result of the preliminary calls; therefore, the tree $G'$ is $(k-i)$-informing. If some person in the communication scheme $G'$ learns more than $k-i$ messages, we can reduce the size of tree $G'$ as follows.

Suppose that as a result of the calls in tree $G'$, person $A$ learned $k-i+s$ meassages ($s\geq 1$), and let $AB$ be the last call made by $A$. Then person $B$ (as well as those who later learned messages from him) became $(k-i+s)$-informed. We perform a reduction: we identify $A$ and $B$ and replace them with a single person $E$ who knows only one message. Let person $E$ make all calls instead of $A$ and $B$, including any preliminary calls, and we cancel the calls $AB$ (both the main and the preliminary one, if it existed). As a result of the main calls, the awareness of the other persons remains at least $k-i$, and the awareness of person $E$ is at least $k-i+s-1\geq k-i$; thus, all persons remains $(k-i)$-informed.

Let us examine what the awareness the persons have after the reduction in the presence of the preliminary calls.

If $AB$ is not a preliminary call, the awareness of all persons is equal to $k$. Indeed, if any person (including $A$ and $B$) learned a message during a preliminary call, then after the reduction that message still reaches this person. Thus, we obtain a communication scheme described in the lemma's statement, in which the value of the parameter $n$ is 1 less.
%
%

If $AB$ is one of the preliminary calls, the final awareness of the persons may decrease but remains at least $k-1$. Indeed, the persons who previously learned both gossips $A$ and $B$, now know the only gossip $E$. We have obtained a case in which $n-1$ persons achieve awareness $k-1$ by means of $i-1$ preliminary calls. The statement of the lemma for this case implies that 
$$
n-1\geq t_{i-2}(k-1)= 2^{k-1-(i-2)-2} +i-2 =t_{i-1}(k)-1,
$$
as we need.

b) Let the tree $G'$ have $j$ edges and $j+1$ vertices. Assume that $\ell$ preliminary calls were made between vertices of $G'$ ($1\leq\ell\leq i$). Consider the connected component of the graph $G$ that contains~$G'$. Since we are proving a lower bound on the number of vertices, we can assume that this component coincides with $G$.

Suppose this component contains $i'$ vertices outside $G'$. We can assume that $i'\ne 0$, otherwise we are dealing with part (a). Then this component contains edges of $G'$ and at least $i'$ additional edges. These edges correspond to preliminary calls, hence $i'+\ell\leq i$. Let us remove from scheme~$G$ these $i'$ calls and all the calls between vertices outside $G'$. We obtain a communication scheme on vertices of $G'$, it has $\ell$ preliminary calls between vertices of graph $G'$, then $j$ calls from scheme~$G'$, and as a result all $j+1$ vertices in graph $G'$ became $(k-i')$-informed. By the part (a) the inequality $j+1\geq t_{\ell-1}(k-i')$ holds and therefore
\begin{align*}
n\geq i'+j+1\geq i'+t_{\ell-1}(k-i') &=i'+2^{k-i'-(\ell-1)-2} + \ell-1
=\\&=
2^{k-(i'+\ell-1)-2}+ (i'+\ell-1) =t_{i'+\ell-1}(k)\geq t_{i-1}(k).
\end{align*}

\kern-\baselineskip\endproof

\begin{lemma}\label{lem:kk2}
Let $n\geq k\geq 4$, $0\leq i\leq k-4$, $j\geq 1$. Let communication scheme $G'$ contain $j$ calls. Let enlarged communication scheme $G$ on $n$ vertices be obtained by appending $i$ preliminary calls to $G'$.

a) If graph $G'$ is a tree, and all its vertices, except one vertex $Z$ are $k$-informed after the calls of scheme $G$, then $n\geq t_{i}(k)$.

b) If graph $G'$ is unicyclic, and all its vertices are $k$-informed after the calls of scheme $G$, then $n\geq t_i(k)$.
\end{lemma}

\proof
a) Let $ZA$ be the last call of person $Z$ in graph $G'$. Immediately after this call, persons $A$ and $Z$ know the same set of messages, but later, $A$ will make additional calls that increase his awareness. As in the proof of Lemma~\ref{lem:k -> k+l}, we glue persons $A$ and $Z$ in graph $G'$, replacing them with single person $E$ who knows one message at the beginning. Let $E$ make all calls instead of $A$ and $Z$, including any preliminary calls, and cancel the calls $ZA$ (both the main call and the preliminary one, if it existed). Then the resulting graph $\widetilde{G}'$ is a $(k-1)$-informing tree.

If none of the preliminary calls is $ZA$, we obtain a communication scheme for $n-1$ people, with $i$ preliminary calls and $j-1$ main calls, such that all the persons from tree $\widetilde{G}'$ has become $(k-1)$-informed. By Lemma~\ref{lem:k -> k+l}b the inequality
$$
n-1\geq t_{i-1}(k-1) = 2^{k-i-2}+i-1 = t_i(k)-1
$$
holds, which is exactly what we need. If one of the preliminary calls is $ZA$, then this scheme contains $i-1$ preliminary calls, and by Lemma~\ref{lem:k -> k+l}b the inequality
$$
n-1\geq t_{i-2}(k-1) = 2^{k-i-1}+i-2 = t_{i-1}(k)-1>  t_i(k)-1.
$$
holds.

b) Let $AB$ be the first call in the cycle of graph $G'$. Consider a new graph on the vertices of $G'$ whose edges are the calls from $G'$ that occurred before the call $AB$. Let $\mathcal A$ be the connected component of this graph that contains vertex $A$. If tree $\mathcal A$ contains more than one vertex, then all its vertices except $A$ are $k$-informed and the required inequality follows from part (a).

Otherwise, $AB$ is the first call of person $A$. Similarly, we can assume that it is the first call of person $B$. Let us consider this call as preliminary, then we obtain a new communication scheme for $n$ people, with $i+1$ preliminary calls and $j-1$ main call, sush that the persons from the tree $G'\setminus\{AB\}$ has become $k$-informed. By Lemma~\ref{lem:k -> k+l}b, the required inequality $n\geq t_{i}(k)$ holds.
\endproof


Observe that for any call sequence, the following permutation rule holds: if $c_1$, \dots, $c_m$, $c'_1$, \dots, $c'_{\ell}$ are consecutive calls in our sequence and the sets of participants in calls $c_1$, \dots, $c_m$ are disjoint from the sets of participants in calls $c'_1$, \dots, $c'_{\ell}$, then we can swap these two groups of calls:
$$
c_1, \dots, c_m, c'_1, \dots, c'_{\ell} \quad\longrightarrow\quad
c'_1, \dots, c'_{\ell}, c_1, \dots, c_m.
$$
The obtained call sequence has the same communication graph and the same outcome in gossip spreading as the original call sequence. Sequences obtained by such an operation are called \emph{equivalent}.

In the following main lemma we say a person is \emph{informed} (or \emph{uninformed}) if, at the moment in question, he already knows $k$ messages (knows less than $k$ messages).

\begin{lemma}\label{lem:2}
Let $n\leq t_{i-1}(k)-1$, where $0\leq i\leq k-4$.
Let $c_1, \ldots, c_{i+j}$ be a sequence of $i+j$ calls made within a group of $n$ persons, where $1\leq j\leq n$. Then:

1) As a result of these calls, at most $j$ persons become informed.

2) If exactly $j$ persons become informed as a result of these calls, then there exists an equivalent sequence $c_1',\dots, c_{i+j}'$ such that the graph of the last $j$ calls $c_{i+1}',c_{i+2}',\ldots, c_{i+j}'$ is a union of connected components of the following forms:

    (i) a unicyclic graph in which all vertices are informed persons;

    (ii) a tree containing exactly one uninformed person, with all other vertices being informed persons.
\end{lemma}

Note that each of these components satisfies Lemma \ref{lem:kk2} and therefore contains ``sufficiently many'' vertices. Hence, under the constraint $n\leq t_{i-1}(k)-1$, the number of such components cannot be large. It seems that only the following cases are possible: either there is a single component (of any of the specified types) or two components (both unicyclic, or one unicyclic and the other a tree). However, we will not need to discuss these details.

\proof Let $G'$ denote the subgraph of the call graph formed by the last $j$ calls.

Observation: the graph $G'$ cannot contain a connected component that is a tree in which all persons are informed. Indeed, in this case, the call scheme consisting of the $i$ preliminary calls and the calls in that component satisfies Lemma \ref{lem:k -> k+l}, and therefore $n\geq t_{i-1}(k)$, which is false.

We prove the statement of the lemma by induction on $j$.

Base case: $1\leq j\leq k-i-2$. Then $i+j\leq k-2$, meaning any connected component of the call graph contains at most $k-1$ vertices; hence, no person has become informed. Both statements turn out to be true but trivial, with the second statement being true ``due to a false premise'' --- which, however, will not hinder reducing the inductive step to it.

Step of induction. Let $j\geq k-i-1$, and for $j'=j-1$ the inductive hypothesis holds.

\medskip
Let us prove statement 1). Suppose that as a result of a sequence of $i+j$ calls, there are $j+1$ informed persons.
Since the last call $c_{i+j}$ produces at most two informed persons, by the time call $c_{i+j-1}$ was completed there were exactly $j-1$ informed persons $x_1$, \dots, $x_{j-1}$ (already knowing $k$ messages each), and the last call took place between the ``new'' informed persons $x_j$ and $x_{j+1}$. By the inductive hypothesis for statement 2), we may assume that all calls $c_{i+1}$, \dots, $c_{i+j-1}$ were between persons $x_1$, \dots, $x_{j-1}$ and, possibly, between persons from the set $\mathcal A$ consisting of persons who did not become informed as a result of those $j-1$ calls. Moreover, each connected component of the communication graph containing $\mathcal A$ is of type (ii).

If $x_{j}\notin \mathcal A$ and $x_{j+1}\notin \mathcal A$ (in particular, if $\mathcal A=\varnothing$), then by swapping this group of calls with $c_{i+j}$, we obtain two informed persons as a result of the $i+1$ calls $c_1$, $c_2$, \dots, $c_i$, $c_{i+j}$, which is impossible by the base of the induction.

If $x_{j}\in \mathcal A$ or $x_{j+1}\in \mathcal A$, then we will show that this case is impossible too.

Indeed, if $x_j\in\mathcal A$ and $x_{j+1}\in\mathcal A$, then after the $j-1$ calls, by the inductive hypothesis, these (still uninformed) persons belong to components of type (ii). Adding an edge between them (as a result of which they become informed) produces a component that is a tree in which all vertices are informed. This is impossible by the observation at the beginning of the proof.

If only one of the persons $x_j$, $x_{j+1}$ belongs to $\mathcal A$, say $x_{j}\in \mathcal A$, then after the $j-1$ calls person $x_j$ was in a component of type (ii), while $x_{j+1}$ did not belong to the graph of those $j-1$ calls. In this case, the call $c_{i+j}$ creates a component that is a tree with all persons informed, which is again impossible.

Thus, the assumption made at the beginning of the argument is false.

\medskip

Let us prove statement 2). Suppose that after $i+j$ calls we have exactly $j$ informed persons. Since $i\leq k-4$, no informed persons can appear as a result of calls $c_1$, $c_2$, \dots, $c_i$. Therefore, every informed person participated in at least one of the calls $c_{i+1}$, \dots, $c_{i+j}$. Clearly, if a call involved a person who already knew $k$ messages or if such a person appeared as a result of this call, then both participants of that call are informed. By the permutation rule, we may assume that the last such call is $c_{i+j}$. Then choose the minimal $p\geq 0$ such that the last calls $c_{i+p+1}$,~\dots, $c_{i+j}$ were made only between informed persons.

If $p=0$, then graph $G'$ has $j$ vertices and $j$ edges, with at least one edge incident to each vertex. By the observation made at the beginning of the proof, no connected component of $G'$ can be a tree. Hence, $G'$ is a union of unicyclic components, which satisfies our requirements (see Fig.~\ref{ris:2unicycle}).

\begin{figure}[h]
\footnotesize\setlength{\unitlength}{1mm}
\begin{minipage}[b]{.45\linewidth}
\begin{center}
\begin{picture}(65,35)(20,-20)
\multiput(60,0)(20,0){2}{\multiput(0,0)(0,-20){2}{\circle*{1}}}
\matrixput(64,8)(20,0){2}(0,-20){2}{\circle*{1}}
\linethickness{.4pt}
\matrixput(60,0)(20,0){2}(0,-20){2}{\put(0,0){\line(1,2){4}}}
\linethickness{1pt}
\multiput(60,0)(0,-20){2}{\put(0,0){\line(1,0){20}}}
\linethickness{2pt}
\multiput(60,0)(20,0){2}{\line(0,-1){20}}
\multiput(20,0)(20,0){2}{\multiput(0,0)(0,-20){2}{\circle*{1}}}
\matrixput(24,8)(20,0){2}(0,-20){2}{\circle*{1}}
\linethickness{.4pt}
\matrixput(20,0)(20,0){2}(0,-20){2}{\put(0,0){\line(1,2){4}}}
\linethickness{1pt}
\multiput(20,0)(0,-20){2}{\put(0,0){\line(1,0){20}}}
\linethickness{2pt}
\multiput(20,0)(20,0){2}{\line(0,-1){20}}
\linethickness{.4pt}
\dashline{1}(40,0)(60,0)
\dashline{1}(40,-20)(60,-20)
\dashline{1}(20,0)(80,-20)
\end{picture}
\end{center}
\caption{
Graph $G'$ (solid lies) and $i$ preliminary calls (dash lines), all persons are $k$-informed, $n=16$, $k=8$, $i=3$, $j=16$, $p=0$.}
\label{ris:2unicycle}
\end{minipage}
\hfil
\begin{minipage}[b]{.45\linewidth}
\begin{center}
\begin{picture}(65,37)(20,-20)
\multiput(60,0)(20,0){2}{\multiput(0,0)(0,-20){2}{\circle*{1}}}
\matrixput(64,8)(20,0){2}(0,-20){2}{\circle*{1}}
\linethickness{.4pt}
\matrixput(60,0)(20,0){2}(0,-20){2}{\put(0,0){\line(1,2){4}}}
\linethickness{1pt}
\multiput(60,0)(0,-20){2}{\put(0,0){\line(1,0){20}}}
\linethickness{2pt}
\multiput(60,0)(20,0){1}{\line(0,-1){20}}
\multiput(20,0)(20,0){2}{\multiput(0,0)(0,-20){2}{\circle*{1}}}
\matrixput(24,8)(20,0){2}(0,-20){2}{\circle*{1}}
\put(57,15){\circle*{1}}
\linethickness{.4pt}
\matrixput(20,0)(20,0){2}(0,-20){2}{\put(0,0){\line(1,2){4}}}
\linethickness{1pt}
\multiput(20,0)(0,-20){2}{\put(0,0){\line(1,0){20}}}
\linethickness{2pt}
\multiput(20,0)(20,0){2}{\line(0,-1){20}}
\linethickness{3pt}
\path(60,0)(57,15)
\linethickness{.4pt}
\dashline{1}(40,0)(60,0)
\dashline{1}(40,-20)(60,-20)
\dashline{1}(20,0)(57,15)
\put(69,-3){\oval(36,43)}
\put(58,15){$V$}\put(77,14){$C$}\put(51.5,6){$c_{i+p}$}
\end{picture}
\end{center}
\caption{
Graph $G'$ (solid lies) and $i$ preliminary calls (dash lines), all persons except $V$ are $k$-infor\-med,
$n=17$, $k=8$, $i=3$, $j=16$, $p=1$.}
\label{ris:2unicyc+tree}
\end{minipage}
\end{figure}

Now assume that $p\geq 1$. By the minimality of $p$, the participants of call $c_{i+p}$ cannot both be informed. But they also cannot both be uninformed, because otherwise we could move this call to the end of the call sequence and obtain the same $j$ informed persons after only $i+j-1$ calls, which contradicts the inductive hypothesis. Therefore, exactly one uninformed person participated in call $c_{i+p}$; denote him by $B$.

Let $G_p'$ be the subgraph of $G$ formed by the calls $c_{i+p}$, \dots, $c_{i+j}$. If $p=1$, then $G_p'=G'$ and it has $j+1$ vertices and $j$ edges. If $G_p'$ is a tree, we have obtained the required configuration. Otherwise (when $p>1$ or $G_p'$ consists of two or more components), denote by $C$ the connected component of $G_p'$ that contains call $c_{i+p}$, see Fig.~\ref{ris:2unicyc+tree}. Let us rename the calls: let $c_1'=c_{i+p}$, $c_2'$,~\dots, $c_r'$ be the calls made within component $C$ in chronological order, and let $c_1''$, $c_2''$, \dots, $c_s''$ be the remaining calls in $G_p'$. In both cases under consideration, we have $r<j$ (and the case $s=0$ is possible when $p>1$).

By the permutation rule, all calls $c_1''$,  \dots, $c_s''$ can be performed before the calls $c_1'$,~\dots, $c_r'$; i.e., the original call sequence is equivalent to sequence 
\begin{equation}
\label{eq:zvonki-ish}
c_1, c_2, \dots, c_{i+p-1}, \quad c_1'',  \dots, c_s'', \quad  c_1', \dots, c_r'.
\end{equation}
Since call $c_1'$ involves only one informed person, component $C$ contains at most $r$ informed persons (it has $r$ edges and therefore at most $r+1$ vertices). Hence, after first $i+j-r$ calls from the sequence \eqref{eq:zvonki-ish}, i.e., after $c_1$, $c_2$,~\dots,~$c_s''$, we have at least $j-r$ informed persons. Then, by the inductive hypothesis, the number of informed persons at that moment is exactly $j-r$ (in this case component $C$ is a tree and contains exactly $r$ informed persons). Moreover, we can reorganize the first $i+j-r$ calls in sequence \eqref{eq:zvonki-ish} to obtain call sequence 
\begin{equation}
\label{eq:zvonki-kon}
\tilde{c}_1, \tilde{c}_2, \dots,  \tilde{c}_{i+j-r},  \quad  c_1', \dots, c_r',
\end{equation}
where the calls $\tilde{c}_{i+1}$, \dots,  $\tilde{c}_{i+j-r}$ form a graph whose connected components are all of type (i) or (ii). (Remark that here some person in a component of type (ii) is uninformed in a sense of the inductive hypothesis, i.e., that after $i+j-r$ calls, he did not become $k$-informed.)

Let us examine, what happens when we perform the calls $c_1', \dots, c_r'$ after the calls $\tilde{c}_{i+1}$, \dots, $\tilde{c}_{i+j-r}$.

Observe that if one of the uninformed persons from a component of type (ii) – call him $A$ – becomes informed as a result of calls $c_1', \dots, c_r'$, this means that after appending edges $c_1', \dots, c_r'$, tree $C$ joined with the component containing $A$ (also a tree), and we obtain a tree with only one uninformed person (call him $B$). The same true if $A=B$. In all other cases, the tree $C$ containing one uninformed person becomes a separate component of the call graph.

Thus, when the calls of sequence \eqref{eq:zvonki-kon} are completed, the connected components have the required form.
\endproof

\subsubsection*{Examples}

Let us demonstrate how the persons can achieve their goal by $P(n,k)$ calls.

E\,x\,a\,m\,p\,l\,e \,1 (\cite[Lemma 2]{ChungTsay}).
Choose two disjoint subsets in the set of all persons: a set $X$ consisting of $i$ persons, and a set $Y=\{y_1, y_2, \dots, y_{2^{k-i-2}}\}$ of $2^{k-i-2}$ persons; see Fig.~\ref{ris:chang-tsay-method}. The calls are organized as follows. First, each person from set $X$ communicates to person $y_1$. As a result, $i$~calls are made, and person $y_1$ now knows $i+1$ messages. Next, by an iterative process we begin doubling the number of ``most informed'' persons in set $Y$. The base: make the calls: $y_1y_2$, $y_3y_4$, $y_1y_3$, $y_2y_4$. After these 4 calls, the first 4 people in set $Y$ each know $i+4$ messages. Inductive step. Suppose at the previous step the first $2^r$ people in set $Y$ (where $r=2$, 3, \dots) knew $i+r+2$ messages. Then at the current step we make $2^r$ calls
\begin{equation}
\label{eqn:zvonki}
y_1y_{1+2^r},  \quad y_2y_{2+2^r}, \quad \dots, \quad y_{2^r}y_{2^{r+1}}
\end{equation}
and now the first $2^{r+1}$ persons in set $Y$ each know $i+r+3$ messages. The process finishes when the set $Y$ is ``covered''. By this point, a total of $i+2^{k-i-2}$ calls have been made, and all persons in $Y$ know $k$ messages. After that, let $y_1$ call every person outside $Y$, making them $k$-informed. Thus, $i+2^{k-i-2}+(n-2^{k-i-2})=n+i$ calls are made.

In this reasoning we used only the inequality $t_i\leq n$.

\begin{figure}[t]
\begin{center}
\setlength{\unitlength}{1mm}
\begin{picture}(75,35)(-1,-26)
\multiput(0,0)(10,0){4}{\circle*{1}}
\multiput(5,-15)(12,0){3}{\circle*{1}}
\matrixput(50,0)(20,0){2}(0,-20){2}{\circle*{1}}
\matrixput(54,8)(20,0){2}(0,-20){2}{\circle*{1}}
\linethickness{.4pt}
\qbezier(30,0)(40,5)(50,0)
\qbezier(20,0)(40,8)(50,0)
\qbezier(10,0)(40,11)(50,0)
\qbezier( 0,0)(40,14)(50,0)
\matrixput(50,0)(20,0){2}(0,-20){2}{\put(0,0){\line(1,2){4}}}
\linethickness{1pt}
\multiput(50,0)(0,-20){2}{\put(0,0){\line(1,0){20}}}
\linethickness{2pt}
\multiput(50,0)(20,0){2}{\put(0,0){\line(0,-1){20}}}
\put(-1,-2){$\underbrace{\qquad\qquad\qquad\qquad\quad}_{\text{$i$ vertices}}$}
\put(48,-22){$\underbrace{\qquad\qquad\qquad\quad}_{\text{$2^{k-i-2}$ vertices}}$}
\put(4,-17){$\underbrace{\qquad\qquad\qquad\quad}_{\text{other verices}}$}
\footnotesize
\put(25,10){$X$}\put(63,10){$Y$}
\put(46,-3){$y_1$}
\end{picture}
\end{center}
\caption{\vtop{\hsize=400pt
En example of call scheme for $n=15$, $i=4$, $k=9$. The vertices in group $X$ make the first calls, the next calls are within group $Y$, and afterwards $y_1$ calls everyone in group $X$ and the ``others'' (these calls are not shown in the diagram). In group $Y$, the thicker the edge, the earlier the call occurs; for edges of the same thickness, the order of calls is arbitrary.
}}
\label{ris:chang-tsay-method}
\end{figure}
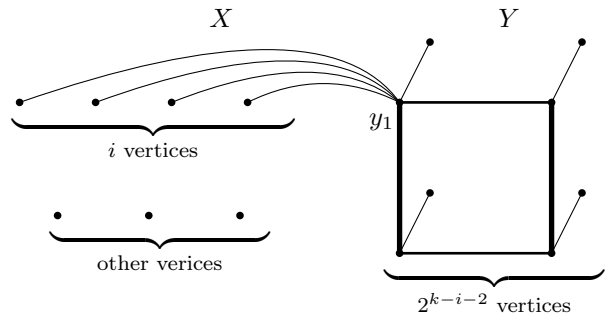

\medskip
E\,x\,a\,m\,p\,l\,e \,2. 
We describe a similar but in some sense more flexible construction for the case $t_i+1\leq n$, such that the graph of the first $n-1$ calls is a tree.

Choose an arbitrary person $A$ and three disjoint subsets: a set $X$ consisting of $i$ persons, a~set $Y=\{y_1, y_2, \dots, y_{2^{k-i-2}}\}$ of $2^{k-i-2}$ persons, and a set $Z$ containing all remaining persons; see Fig.~\ref{ris:kk-method}. The calls are organized as follows.

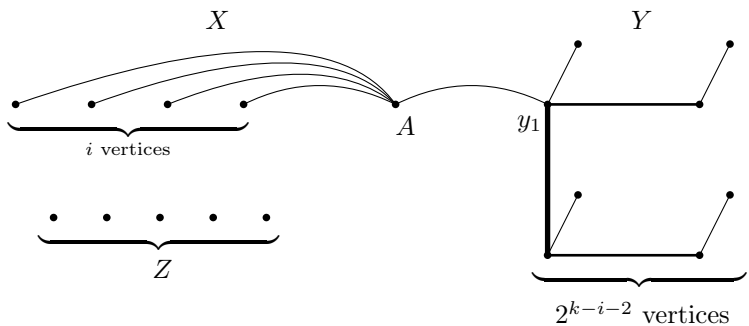
\begin{figure}[h]
\begin{center}
\footnotesize 
\setlength{\unitlength}{1mm}
\begin{picture}(95,35)(-1,-26)
\multiput(0,0)(10,0){4}{\circle*{1}}
\multiput(5,-15)(7,0){5}{\circle*{1}}
\put(50,0){\circle*{1}}
\matrixput(70,0)(20,0){2}(0,-20){2}{\circle*{1}}
\matrixput(74,8)(20,0){2}(0,-20){2}{\circle*{1}}
\linethickness{.4pt}
\qbezier(30,0)(40,5)(50,0)
\qbezier(20,0)(40,8)(50,0)
\qbezier(10,0)(40,11)(50,0)
\qbezier( 0,0)(40,14)(50,0)
\qbezier(50,0)(60,5)(70,0)
\matrixput(70,0)(20,0){2}(0,-20){2}{\put(0,0){\line(1,2){4}}}
\linethickness{1pt}
\multiput(70,0)(0,-20){2}{\put(0,0){\line(1,0){20}}}
\linethickness{2pt}
\put(70,0){\line(0,-1){20}}
\put(-1,-2){$\underbrace{\qquad\qquad\qquad\qquad\quad}_{\text{$i$ vertices}}$}
\put(68,-22){$\underbrace{\qquad\qquad\qquad\qquad}$}
\put(71,-29){$2^{k-i-2}$ vertices}
\put(3,-17){$\underbrace{\qquad\qquad\qquad\qquad\quad}$}
\put(18,-23){$Z$}
\put(25,10){$X$}
\put(81,10){$Y$}
\put(50,-4){$A$}\put(66,-3){$y_1$}

\end{picture}
\end{center}
\caption{\vtop{\hsize=390pt Example of a call scheme for $n=18$, $i=4$, $k=9$. The vertices in group $X$ make the first calls, next call is $Ay_1$, then calls within group $Y$, and afterwards $y_1$ calls everyone in group $X$, $Z$ and person $A$. In group $Y$, the thicker the edge, the earlier the call occurs; for edges of the same thickness, the order of calls is arbitrary.
}}
\label{ris:kk-method}
\end{figure}

First, each person from $X$ communicates his message to person $A$; then $A$ calls $y_1$. As a result, $i+1$ calls are made, and now $A$ and $y_1$ know $i+2$ messages each. Next, by an iterative process (but without the 4-cycle, whose role is taken over by $A$), we begin doubling the number of ``most informed'' persons in $Y$. If at the previous step the first $2^r$ people (where $r=0$, 1, \dots) in set $Y$ knew $i+r+2$ messages, then at the current step we make $2^r$ calls according to scheme \eqref{eqn:zvonki}, and now the first $2^{r+1}$ persons in $Y$ each know $i+r+3$ essages. As in the previous example, once the set $Y$ is ``covered'', every person in it knows $k$ messages. Then $y_1$ calls all the persons in~$Z$. By this point $n-1$ calls have been made, the communication graph now is a tree on $n$ vertices. To complete the gossip spreading, person $y_1$ calls everyone in group $X$ and person $A$ --- this requires an additional $i+1$ calls.

This construction illustrates statement a) of Lemma~\ref{lem:kk2}: the graph $G'$ consists of vertex $A$ together with the entire component $Y$.

\medskip
E\,x\,a\,m\,p\,l\,e \,3. 
The construction of Example 2 admits the following modification. Let us write $Y_1$ instead of $Y$. Once set $Y_1$ is ``covered'' and all persons in it know $k$ messages each, we can try (if the inequality $n< t_{i-1}(k)$ permits) to select from group $Z$ another subset, $Y_2$, consisting of $2^{k-i-3}$ persons; see Fig.~\ref{ris:kk-ext-method}. Person $A$ calls some person $y_2\in Y_2$; as a result, both know $i+3$ messages each. Afterwards, a binary process, analogous to the one in $Y_1$, is launched within $Y_2$. After that (if the inequality $n< t_{i-1}(k)$  still allows), we can select yet another subset, $Y_3$, and so on. Finally, $y_1$ calls the remaining members of group $Z$; the communication graph becomes a tree on $n$ vertices. Then $y_1$ makes an additional $i+1$ calls  --- calling everyone in group $X$ and person $A$. In total, $n+i$ calls are made.

\begin{figure}[h]
\begin{center}
\setlength{\unitlength}{.7mm}
\begin{picture}(135,45)(-1,-26)
\multiput(0,0)(10,0){4}{\circle*{1}}
\put(23,-15){\circle*{1}}
\put(50,0){\circle*{1}}
\matrixput(70,0)(20,0){2}(0,-20){2}{\circle*{1}}
\matrixput(74,8)(20,0){2}(0,-20){2}{\circle*{1}}
\matrixput(115,0)(20,0){2}(0,-20){2}{\circle*{1}}
\linethickness{.4pt}
\qbezier(30,0)(40,5)(50,0)
\qbezier(20,0)(40,8)(50,0)
\qbezier(10,0)(40,11)(50,0)
\qbezier( 0,0)(40,14)(50,0)
\qbezier(50,0)(60,5)(70,0)
\qbezier(50,0)(82,35)(115,0)
\matrixput(70,0)(20,0){2}(0,-20){2}{\put(0,0){\line(1,2){4}}}
\linethickness{1pt}
\multiput(70,0)(0,-20){2}{\put(0,0){\line(1,0){20}}}
\multiput(115,0)(0,-20){2}{\put(0,0){\line(1,0){20}}}
\linethickness{2pt}
\put(70,0){\line(0,-1){20}}
\put(115,0){\line(0,-1){20}}
\put(-1,-2){$\underbrace{\qquad\qquad\qquad\;}_{\text{$i$ vertices}}$}
\put(68,-22){$\underbrace{\qquad\qquad\quad}_{\text{$2^{k-i-2}$ vertices}}$}
\put(113,-22){$\underbrace{\qquad\qquad\quad}_{\text{$2^{k-i-3}$ vertices}}$}
\put(17,-17){$\underbrace{\qquad}$}
{\footnotesize
\put(22,-25){$Z$}
\put(25,10){$X$}
\put(81,10){$Y_1$}\put(126,10){$Y_2$}
\put(50,-4){$A$}\put(64,-3){$y_1$}\put(109,-3){$y_2$}
}

\end{picture}
\end{center}
\caption{\vtop{\hsize=375pt Example of call scheme for $n=18$, $i=4$, $k=9$.
First, we make the calls from the vertices in group $X$, then call $Ay_1$, then the calls within group $Y_1$, then call $Ay_2$, then the calls within group $Y_2$, and finally, $y_1$ calls everyone in groups $X$, $Z$, and person $A$.
}}
\label{ris:kk-ext-method}
\end{figure}
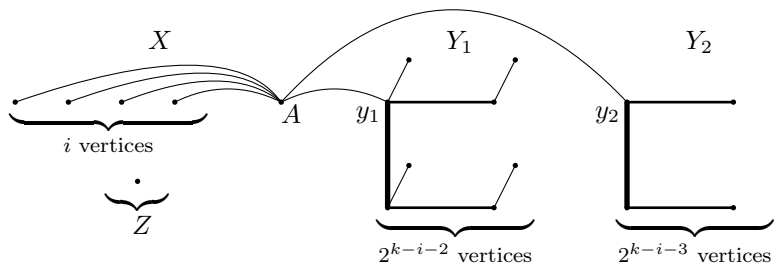

\subsubsection*{Proof of the theorem}

Statement 1 of Theorem \ref{thm:chung-tsay} is proved in \cite{ChungTsay}. Statement 2 follows from Lemmas \ref{lem:1}–\ref{lem:k -> k+l} and \ref{lem:2} (the bound) and from Example 1 (the construction), which we cited from~\cite{ChungTsay}.

\end{document}